\documentclass[12pt]{amsproc}
\usepackage{amsmath,amsthm}
\usepackage{amsfonts,amssymb}
\newtheorem{thm}{Theorem}[section]
\newtheorem{cor}[thm]{Corollary}
\newtheorem{lem}[thm]{Lemma}
\newtheorem{prop}[thm]{Proposition}

\theoremstyle{remark}

\def\CF{{\mathcal F}}

\def\CH{{\mathcal H}}

\def\CR{{\mathcal R}}
\def\CL{{\mathcal L}}

\def\h{{\mathfrak h}}

\def\C{{\mathbb C}}
\def\H{{\mathbb H}}
\def\N{{\mathbb N}}
\def\R{{\mathbb R}}

\def\Z{{\mathbb Z}}

\def\1{\text{\bf {1}}}

\def\h{{\mathfrak h}}

\begin{document}

\title[Riesz transforms]
{Revisiting Riesz transforms on Heisenberg groups}
\author{P. K. Sanjay}

 \author{ S. Thangavelu}

\address{Department of Mathematics\\ Indian Institute
of Science\\Bangalore-560 012}
\email[Sanjay P. K.]{sanjay@math.iisc.ernet.in}
 \email[S. Thangavelu]{veluma@math.iisc.ernet.in}
 \address{Permanent address(Sanjay P. K.):Department of Mathematics, National Institute of Technology, Calicut- 673 601}

\date{\today}
\thanks{ }
\maketitle

\begin{center}
Dedicated to Eli Stein on his eightieth birthday
\end{center}

\begin{abstract}
We characterise higher order Riesz transforms on the Heisenberg group and 
also show that they satisfy dimension-free bounds under some assumptions on 
the multipliers. Using transference theorems, we deduce boundedness theorems
for Riesz transforms on the reduced Heisenberg group and hence also for 
the Riesz transforms associated to multiple Hermite and Laguerre expansions.

\end{abstract}


\section{Introduction}
\setcounter{equation}{0}

The aim of this paper is twofold: to prove dimension free estimates for Riesz
transforms associated to reduced Heisenberg groups $ \H^n_{red} $ and to
study Riesz transforms of higher order on Heisenberg groups $ \H^n $ which
are in a sense canonical. Dimension free estimates for Riesz transforms on
$ \H^n $ have been studied in the literature by Coulhon et al. \cite{CMZ}
but the proof
given there, which depends very much on the dilation structure of $ \H^n $,
does not work for the reduced Heisenberg group. However, we can view the 
Riesz transforms on the Heisenberg group (reduced Heisenberg group) as an 
operator valued multiplier for the Fourier transform (resp. Fourier series). 
Hence using a transference theorem we can deduce results for the reduced 
Heisenberg group from those for $ \H^n.$ Let us set up the notation before 
stating our results.

Let $ \H^{n} = \C^n \times \R$ denote the (2n+1)-dimensional Heisenberg group
with the group operation
$(z,t)(w,s)=(z+w, t+s+\frac{1}{2}Im(z.\bar{w}))$. Its Lie algebra $\h_n$ is
generated by the $(2n+1)$ left invariant vector fields
$$ T=\frac{\partial}{\partial t},
 X_j=\left(\frac{\partial}{\partial x_j}+\frac{1}{2}y_j
\frac{\partial}{\partial t}\right), Y_j=\left(\frac{\partial}{\partial y_j}
-\frac{1}{2}x_j \frac{\partial}{\partial t}\right) $$
$ (j=1,2,\cdots n).$ The operator
$ \mathcal{L}= -\sum^n_{j=1}(X_{j}^{2}+Y_{j}^{2})$ is called  the sublaplacian
on $H^n $. Written explicitly
$$ \mathcal{L} =-\Delta_{z}-\frac{1}{4}|z|^2 \frac{\partial^2}{\partial t^2}
+N \frac{\partial}{\partial t} $$ where $\Delta_z$ is the
ordinary Laplacian on $\C^n$ and
$$ N=\sum^n_{j=1} \left(x_i\frac{\partial}{\partial y_j}-y_i \frac{\partial}{\partial x_j}\right).$$
We can also write $ \mathcal{L}$ as
$$ \mathcal{L}= -\frac{1}{2}\sum^n_{j=1}(Z_j \bar{Z}_j + \bar{Z}_j Z_j)$$
where $Z_j=(X_j - i Y_j)$ and $\bar{Z}_j=(X_j+i Y_j)$ for
$j=1,2,\cdots n$. We define the Riesz transforms
$ R_j=Z_j \mathcal{L}^{-1/2}$ and
$\bar{R}_j=\bar{Z}_j \mathcal{L}^{-1/2},j=1,2,\cdots,n$  associated with
$\mathcal{L}$.

These operators are known to be singular integral operators on $ \H^n $ and
hence bounded on $ L^p(\H^n), 1 < p < \infty,$ and of weak type $(1,1).$
Moreover, in \cite{CMZ} it is proved that the bounds do not depend on the
dimension of $ \H^n.$ In this work we consider the same operators acting on
the reduced Heisenberg group $ \H^n_{red} = \H^n/\Gamma $ where $ \Gamma =
\{ (0,2\pi k): k \in \Z \} $ is a central subgroup. Thus functions on
$ \H^n/\Gamma$ are $ 2\pi $ periodic in the central variable.  For the Riesz
transforms acting on $ L^p(\H^n/\Gamma) $ we prove the following theorem.

\begin{thm} For each $1<p<\infty$, there exists a constant $C_p$ independent
of the dimension n, such that for all $f\in L^p(H^n/\Gamma),$
$$ \left\| \left(\sum^n_{j=1}|R_j f|^2+
\sum^n_{j=1}|\bar{R}_j f|^2\right)^{1/2} \right\| _{p} \leq C_p
\| f \|_p.$$
\end{thm}

As we mentioned earlier, we deduce this theorem from the result of 
\cite{CMZ} on the Heisenberg group. The appropriate transference theorem
is stated in Section 2.

We now turn our attention to higher order Riesz transforms on the Heisenberg
group. Observe that the Riesz transforms defined as $ (X_j-iY_j)\CL^{-1/2},
(X_j+iY_j)\CL^{-1/2} $ are multipliers for the group Fourier transform. The
corresponding multipliers are given by $  A_j(\lambda)H(\lambda)^{-1/2} $ and
$  A_j^*(\lambda)H(\lambda)^{-1/2} $ respectively where for
$ j = 1,2,...,n $ and $ \lambda \in \R , A_j(\lambda) $ and $
A_j^*(\lambda) $ are the creation and annihilation
operators:
$$ A_j(\lambda) = -
\frac{\partial}{\partial \xi_j}+\lambda \xi_j ,
A_j^*(\lambda) = \frac{\partial}{\partial \xi_j}+\lambda \xi_j $$
and
$$ H(\lambda) = -\Delta+\lambda^2 |x|^2
= \frac{1}{2}
\sum_{j=1}^n \left( A_j(\lambda)A_j^*(\lambda)+A_j^*(\lambda)A_j(\lambda)\right) $$ is the Hermite operator. We would like to consider higher order Riesz
transforms on $ \H^n $ as Fourier multipliers corresponding to higher order
analogues of $  A_j(\lambda)H(\lambda)^{-1/2} $ and
$  A_j^*(\lambda)H(\lambda)^{-1/2} .$ In analogy with the case of standard
Laplacian on $ \R^n $ the above operators can be considered as analogues of
Riesz transform for the Hermite expansions.
Moreover, as shown in \cite{T0} they turn out
to be pseudodifferential operators of order zero and hence bounded on
$ L^p(\R^n), 1 < p < \infty.$

Higher order Riesz transforms associated to the Hermite operator have also
been studied in the literature  by defining them as operators of the form
$ A(\lambda)^\alpha H(\lambda)^{-\frac{1}{2}|\alpha|} $ and
$ A^*(\lambda)^\alpha H(\lambda)^{-\frac{1}{2}|\alpha|} $ where $ \alpha $
is a multiindex and $ A(\lambda) = ( A_j(\lambda)),
A^*(\lambda) = ( A_j^*(\lambda)) $ are vectors, see \cite{HRST} and \cite{LP}.
Here we consider the operators  $ G_\lambda(P)H(\lambda)^{-\frac{m}{2}} $ as 
the natural candidates for the higher order Riesz transforms, where $ P $ is a 
bigraded solid harmonic on $ \C^n $ and  $ G_\lambda(P)$ is the operator associated to $ P $ by the Weyl correspondence. (For the definition of 
$ G_\lambda(P)$ as well as the Weyl transform $ W_\lambda(P)$ we refer to Section 3; more details can be found \cite{T2}.) This class of operators include 
the preceding ones, because $ P(z) = z^\alpha $ and $ Q(z) = \bar{z}^\alpha $ 
are  bigraded solid harmonics of bidegrees $ (|\alpha|,0) $ and 
$ (0,|\alpha|),$ and by a theorem of Geller \cite{G}, $ G_\lambda(P) = 
A(\lambda)^\alpha $ and $ G_\lambda(Q) = A^*(\lambda)^\alpha.$

\begin{thm} For every bigraded solid harmonic $ P $ of total degree $ m $
the operator $  G_\lambda(P)H(\lambda)^{-\frac{m}{2}} $ is bounded
on $ L^p(\R^n), 1 < p < \infty  $ and is of weak type (1,1).
\end{thm}

This theorem can be deduced, in principle, from the known results on higher
order Riesz transforms. However, we prefer to give a painless proof which is
based  on Mauceri's transference theorem \cite{M} for Weyl multipliers. In
view of this transference, we only need to show that
$ G_\lambda(P)H(\lambda)^{-\frac{m}{2}} $ is an $ L^p $ multiplier for the
Weyl transform $ W_\lambda.$ The operator $ T_P $ defined on $ L^2(\C^n) $ by
$$ W_\lambda(T_Pf) = W_\lambda(f)G_\lambda(P)H(\lambda)^{-\frac{m}{2}} $$
turns out to be a twisted convolution operator with a singular kernel. Thanks
to the Hecke-Bochner type formula for the Weyl transform, due to Geller
\cite{G}, the
kernel can be estimated easily. By appealing to the theory of oscillatory
singular integrals developed by Ricci and Stein \cite{RS} and Chanillo and Christ
\cite{CC} we can prove

\begin{thm}For any bigraded solid harmonic $ P $ of total degree $ m $
the operator $ T_P $ defined above is bounded on
$ L^p(\C^n), 1 < p < \infty  $ and is of weak type (1,1).
\end{thm}

We now consider the Riesz transforms
$  G_\lambda(P)H(\lambda)^{-\frac{m}{2}} $
as multipliers for the (group) Fourier transform on the Heisenberg group $
\H^n $ and explain why they are natural. Recall that a Fourier multiplier for
the Heisenberg group is a family of  bounded linear operators $ M(\lambda) $
on $ L^2(\R^n) $ and the multiplier transformation $ T_M $ is defined by
$ \pi_\lambda(T_Mf) = \hat{f}(\lambda)M(\lambda) $ where $ \hat{f}(\lambda)
= \pi_\lambda(f)$ is the group Fourier transform on the Heisenberg group. Such operators are
precisely those which are given by convolution with certain kernels on $ \H^n $ and hence commute with translations. The unitary group $ U(n) $ acts on $
\H^n $ by automorphisms which leads to an action on $ L^2(\H^n).$ Let
$\CH_{p,q}$ stand for the space of bigraded solid harmonics of bidegree
$(p,q), $ which supports an irreducible unitary representation $ R(\sigma) $ of $ U(n).$
In what follows we let $ P_k(\lambda) $ stand for the orthogonal projection of
$ L^2(\R^n) $ onto the $k-$th eigenspace of $ H(\lambda).$ We also denote by $
\rho(\sigma) f $ the action of $ \sigma \in U(n) $ on functions defined on $ \H^n $,
i.e. $ \rho(\sigma) f(z,t) = f(\sigma^{-1}z,t) $ and by $ \delta_r, r > 0 $
the nonisotropic dilation $ \delta_r f(z,t) = f(rz,r^2t).$

\begin{thm} Let $ T $ be a translation invariant operator taking $ L^2(\H^n) $
into $ L^2(\H^n,\CH_{p,q})$ and let $ M(\lambda) $ be the corresponding Fourier multiplier. Assume that (i)$ R(\sigma)Tf(z,t) =
\rho(\sigma) T \rho(\sigma^{*})f(z,t) $ for every $ \sigma \in U(n) $ (ii)
$ T\delta_r f(z,t) = \delta_r Tf(z,t) $ for every $ r > 0 $ and (iii)
$ M(\lambda)P_k(\lambda) = ((2k+n)|\lambda|)^{-\frac{1}{2}(p+q)}S(\lambda)$
for some (unbounded) operator $S(\lambda).$ Then for any linear functional
$ \beta $ of $ \CH_{p,q}$ the operator $ \beta(T)f = \beta(Tf) $ is a linear
combination of the Riesz transforms with multipliers $
G_\lambda(P)H(\lambda)^{-\frac{p+q}{2}} $ as $ P $ runs through an orthonormal
basis of $ \CH_{p,q}.$
\end{thm}

The above theorem is the analogue of a result of Stein for higher order Riesz
transforms on $ \R^n $, see page 79 in \cite{S1}. For variations on the same theme we
refer to the recent works \cite{KN1} and \cite{KN2} of Kobayashi and Nilsson.
The operators on $ L^2(\H^n) $ with multipliers $
G_\lambda(P)H(\lambda)^{-\frac{p+q}{2}} $ turn out to be singular integral operators. Hence we get

\begin{thm} Let $ T $ satisfy the assumptions of the previous theorem and let
$ \beta $ be a linear functional on $\CH_{p,q}.$ Then $ \beta(T) $ is bounded
on $ L^p(\H^n), 1 < p < \infty $ and is of weak type (1,1).
\end{thm}

Given $ P \in \CH_{p,q} $ let $ R_P $ stand for the (higher order) Riesz
transform with multiplier $ G_\lambda(P)H(\lambda)^{-\frac{p+q}{2}}.$
The boundedness of these higher order Riesz transforms on  $ L^p(\H^n) $
is well known, see e.g. the works of Folland \cite{F}, Lohou\'e and Varopoulos
\cite{LV} and Ter Elst et al. \cite{ERS}. However, dimension-free estimates
are not known and we conjecture that such estimates are true. Here we prove
such estimates under some assumptions on $ P.$ Indeed, we let $ P_0(z) =
z_j^p\bar{z_k}^q , j \neq k $ and denote by  $ O(P_0) $ the orbit of $ P_0 $
under the action of $ U(n).$ We then have

\begin{thm} For every $ P \in O(P_0) $ the Riesz transform $ R_P $ satisfies
the estimate $ \|R_Pf\|_p \leq C_p \|f\|_p $  on $ L^p(\H^n), 1 < p < \infty$
where $ C_p $ is independent of the dimension $ n $ and 
$ P  \in O(P_0) .$
\end{thm}

We actually show that for any solid harmonic $ P$  and $ \sigma \in U(n) $ 
the operators $R_P$ and $ R_{\rho(\sigma)P} $ have the same norm on 
$ L^p(\H^n).$ The theorem is then proved by showing that $ R_{P_0} $ 
satisfies dimension-free bounds on $ L^p(\H^n).$ We prove this using known 
estimates for
first order Riesz transforms and a dimension-free estimate for a particular
singular integral operator on the Heisenberg group, see Section 3.

A slight strengthening of the above theorem is possible. ( We are thankful 
to the referee for pointing this out.) Given a solid harmonic 
$ P $ on $ \C^n $ 
we can also consider it as a solid harmonic on $ \C^m $ for any $ m \geq n $ 
which depends only on the first $n$ variables. Hence we can define Riesz 
transforms $ R_P^m $ acting on $ L^p(\H^m).$ As $ U(n) $ can be considered as 
a subgroup of $ U(m) $ in the obvious way, the operator norm of $ R_P^m $ on 
$ L^p(\H^m) $ is the same as that of $ R_{P_0}^m $ which is independent of 
$m.$ Thus $ R_P^m :L^p(\H^m) \rightarrow L^p(\H^m) $ satisfies estimates 
which are independent of $ m \geq n \geq 2  $ and $ P \in O(P_0).$

By considering functions on $ \C^n $ which are homogeneous we show that
the Riesz transforms $ T_{z_j} $ and $ T_{\bar{z_j}} $ are related to
Riesz transforms for multiple Laguerre expansions studied by Nowak and
Stempak \cite{NS1}. We prove that the boundedness of these
Laguerre-Riesz  transforms are equivalent to certain weighted norm
inequalities for $ T_{z_j} $ and $ T_{\bar{z_j}} $ on $ L^p(\C^n).$

\section{Riesz transforms on the reduced Heisenberg group}
\setcounter{equation}{0}

In this section we prove Theorem 1.1. As we mentioned earlier this will be 
done by using a transference theorem which we now describe. Let the special 
Hermite operators $ L_\lambda $ be defined by the equation $ \CL(e^{i\lambda t}
f(z)) = e^{i\lambda t}L_\lambda f(z).$ Then it is known that
$$ L_\lambda = -\Delta+\lambda^2 |z|^2+i\lambda N $$ 
where $ \Delta $ is the Laplacian on $ \C^n $ and 
$$  N=\sum^n_{j=1} \left(x_i\frac{\partial}{\partial y_j}-y_i 
\frac{\partial}{\partial x_j}\right).$$
We can also write $ L_\lambda $  as 
$$ L_\lambda = -\frac{1}{2}\sum_{j=1}^n \bigg(Z_j(\lambda)\bar{Z}_j(\lambda)+
 \bar{Z}_j(\lambda)Z_j(\lambda)\bigg) $$
where $ Z_j(e^{i\lambda t}f(z))= e^{i\lambda t}Z_j(\lambda)f(z) $ and $ 
 \bar{Z_j}(e^{i\lambda t}f(z))= e^{i\lambda t}\bar{Z_j}(\lambda)f(z).$ In 
view of this the Riesz transforms on the Heisenberg group can be written as 
$$ R_jf(z,t) = \frac{1}{2\pi} \int_{-\infty}^\infty e^{-i\lambda t} 
R_j(\lambda)f^\lambda(z) 
d\lambda $$
where $ R_j(\lambda) = Z_j(\lambda)L_\lambda^{-1/2} .$
We have a similar expression for $ \bar{R_j}.$  

Therefore, we can view the Riesz transforms as operator valued 
multipliers for the (Euclidean) Fourier transform on $ \R.$ Indeed, if we let 
$ X = L^p(\C^n) $ then $ L^p(\H^n) $ can be identified with $ L^p(\R,X)$ and 
$ R_j $ are multiplier transforms corresponding to the (operator valued) 
multipliers $ R_j(\lambda).$ Using the same notation for Riesz transforms on $
\H^n/\Gamma $ they are given by
$$ R_jf(z,t) = \frac{1}{2\pi} \sum_{k=-\infty}^\infty e^{-i k t} 
R_j(k)f^k(z).$$ Thus the Riesz transforms for $\H^n/\Gamma $ are multipliers 
for the Fourier series acting on $ L^p(S^1,X).$ In order to treat the vector 
of the Riesz transforms we set
$$ \CR f =(R_1f,R_2f,...,R_nf,\bar{R_1}f,\bar{R_2}f,...,\bar{R_n}f) $$
and introduce the Banach space $ Y = L^p(\C^n)\times L^p(\C^n).....\times 
L^p(\C^n) $ ($2n$ times) equipped with the norm
$$ \|f\|_Y = \| (\sum_{j =1}^{2n}|f_j|^2)^{1/2}\|_p,~~~~ 
f = (f_1,......,f_{2n}).$$ Then the  result of \cite{CMZ} can be written as 
$$ \|\CR f\|_{L^p(\R,Y)} \leq C \|f\|_{L^p(\R,X)} $$ 
where $ C $ is independent of the dimension. We are interested in showing 
that a similar estimate holds when $ L^p(\R,X) $ and $L^p(\R,Y) $ are 
replaced by $ L^p(S^1,X)$ and $ L^p(S^1,Y)$ respectively. All we need in 
order to prove this is the following transference result.

Denote by $ (L^p(\R,X),L^p(\R,Y)) $ the set of all Fourier multipliers taking 
$L^p(\R,X)$ into $ L^p(\R,Y).$ Let $ T \in (L^p(\R,X),L^p(\R,Y)) $ and let 
$ m $ be  the corresponding multiplier so that 
$$ Tf(t) = \frac{1}{2\pi} \int_{-\infty}^\infty e^{-i\lambda t} m(\lambda)
\hat{f}(\lambda) d\lambda .$$ We can define a periodised operator $ \tilde{T} $ by 
$$ \tilde{T}f(t) = \frac{1}{2\pi} \sum_{k=-\infty}^\infty e^{-ikt}m(k)
\hat{f}(k)$$
for every $ f \in L^p(S^1,X).$

\begin{thm}  Suppose $ 1 \leq p \leq \infty $ and $ T \in 
(L^p(\R,X),L^p(\R,Y)) $ where $ X $ and $ Y $ are Banach spaces. Let $ m $ be 
the multiplier corresponding to $ T $ and assume that $ m $ is continuous at 
every point of $ \Z, $ the set of all integers. Then the periodised operator 
$\tilde{T} $ belongs to $ (L^p(S^1,X),L^p(S^1,Y)) $ and satisfies 
$ \|\tilde{T}\| \leq \|T\|.$
\end{thm}
 
When $ X = Y = \C $ this theorem is due to de Leeuw and a proof can be 
found in Stein-Weiss \cite{SW}, see 
Theorem 3.8, Chapter VII. The proof given there can be easily modified to 
yield the above 
version of transference. We leave the details to the interested reader. Once 
we have the above theorem it is clear that Theorem 1.1 follows from the 
corresponding result for Riesz transforms on $ \H^n $ proved in \cite{CMZ}.
We can also treat higher order Riesz transforms: for every 
$ P \in \CH_{p,q} $ let us define
$$ \tilde{R}_Pf(z,t) = \frac{1}{2\pi}\sum_{k=-\infty}^\infty e^{-i k t}
T_Pf^k(z) $$ where $ T_P $ is as in Theorem 1.3. Then

\begin{thm} The operators $ \tilde{R}_P $ defined above are all bounded on 
$ L^p(\H^n/\Gamma) $ for $ 1 < p < \infty.$
\end{thm}

The first order Riesz transforms satisfy dimension-free bounds and we expect 
the same for higher order Riesz  transforms. This is an interesting open problem even in the case of $ \H^n.$ Our partial result, Theorem 1.6, naturally 
has an analogue for $ \H^n/\Gamma.$

\section{ Higher order Riesz transforms on the Heisenberg group}
\setcounter{equation}{0}

In this section we prove Theorems 1.4, 1.5 and 1.6 stated in the introduction.
The statement and the proof of Theorem 1.4 are both inspired by
the corresponding result of Stein for the
Euclidean case stated in \cite{S1} (see Chapter III, Section 4.8, page 79).
For the proof we need some facts about
Weyl transform $ W_\lambda $ and Weyl correspondence $ G_\lambda.$ We closely
follow the notations from the monograph \cite{T2} to which we refer for any
unexplained terminology. For each non-zero real $ \lambda $ the  Weyl
transform  $ W_\lambda(f) $ of a function on $ \C^n $ is defined as the
operator
$$ W_\lambda(f) = \int_{\C^n} f(z) \pi_\lambda(z,0)~~ dz $$ where
$ \pi_\lambda $ is the Schr\"{o}dinger representation of $ \H^n $ with
parameter $ \lambda. $ The Weyl correspondence is then defined by
$ G_\lambda(f) = W_\lambda(\CF_\lambda(f)) $ where $ \CF_\lambda(f) $ is
the symplectic Fourier transform of $ f .$ As in the
introduction we denote by $\CH_{p,q} $ the space of bi-graded spherical
harmonics of bi-degree $ (p,q).$

We now recall the Hecke-Bochner formula for the Weyl transform proved
by Geller \cite{G}. We remark that the notation we use here is different from
that of Geller as we follow \cite{T2}. When $ f $ is radial and $ P $ is a
bigraded solid harmonic of bidegree $ (p,q) $ we have (for
$ \lambda > 0$)
$$ W_\lambda(Pf) = G_\lambda(P)\left( \sum_{k=p}^\infty R_{k-p}^\lambda(f)
P_k(\lambda) \right) $$
where $ P_k(\lambda) $ are the projections associated to the Hermite operator
$ H(\lambda) $ and
$$ R_k^\lambda(f) = \frac{\Gamma(k-p+1)\Gamma(n)}{\Gamma(k+q+n)} \int_{\C^{n+p+q}} f(|z|)
\varphi_{k,\lambda}^{n+p+q-1}(z) dz.$$
In the above formula
$$ \varphi_{k,\lambda}^{n+p+q-1}(z) = L_k^{n+p+q-1}\left(\frac{1}{2}\lambda|z|^2\right)
e^{-\frac{1}{4}\lambda|z|^2} $$ are the Laguerre functions of type
$ (n+p+q-1).$ We refer to \cite{T2} (Theorem 2.6.2) and \cite{G}
(Theorem 4.2) for more details.

We now begin with a proof of Theorem 1.4. Since $T$ is a translation invariant
operator taking $ L^2(\H^n) $ into $ L^2(\H^n,\CH_{p,q})$ it is a convolution
operator with a distributional kernel  $k(z,t)$ taking values in
$\CH_{p,q} $: $ Tf(z,t) =
f*k(z,t).$ Let $ \{ Y_j: j = 1,2,...,d(p,q) \}$ be an orthonormal basis of
$\CH_{p,q} $ consisting of spherical harmonics so that
$$ k(z,t) = \sum_{j=1}^{d(p,q)} k_{j}(z, t)Y_j $$
where $k_j(z, t)$ are scalar valued distributions on $\H^{n}.$
Then, we can write
$$ Tf(z, t) = \sum_{j =1}^{d(p,q)} T_jf(z, t)Y_{j} = \sum_{j=1}^{d(p,q)}
f*k_j(z,t) Y_j .$$
The hypothesis (i), namely $ \rho(\sigma) T\rho(\sigma^{-1})f = R(\sigma)Tf $
translates into
$$ R(\sigma)Tf(z,t) = T \rho(\sigma^{-1}) f(\sigma^{-1}z,t) $$
$$ =  \sum_{j=1}^{d(p,q)} T_{j} \rho(\sigma^{-1})f(\sigma^{-1}z, t)Y_{j}.$$
Let $ (a_{i,j}(\sigma)) $ stand for the matrix corresponding to $ R(\sigma)$
in the basis $ \{ Y_j: j=1,2,..,d(p,q) \}$. Then we have
$$ \sum_{j=1}^{d(p,q)} a_{i,j}(\sigma) T_{j}f(z,t) =
T_{i}\rho(\sigma^{-1})f(\sigma^{-1}z, t) $$
which gives after a simple calculation the relation
$$  f * \sum_{j=1}^{d(p,q)} a_{i,j}(\sigma) k_j(z,t) =
f * \rho(\sigma^{-1})k_i(z,t).$$
This shows that
$$ R(\sigma)k(z,t) = \rho(\sigma^{-1})k(z,t) = k(\sigma z,t),~~~ \sigma \in U(n).$$
From this relation we observe that for any unit vector $ w \in \C^{n}, k(w,t)
$ as an element of $ \CH_{p,q} $ is invariant under all $ \sigma \in U(n) $
which fix $ w.$ Therefore, there is a scalar valued function $ c(w,t) $ such
that $ k(w,t) = c(w,t)Y_w $ where $ Y_w $ is the zonal harmonic with pole at
$w.$

For any two unit vectors $ z $ and $ w $ we can always find $ \sigma \in
U(n) $ such that $ \sigma z = w.$ Hence $ R(\sigma)k(z,t) = k(\sigma z,t )
= k(w,t) $ leads to the equation
$$ c(w,t)Y_w = R(\sigma)k(z,t) = c(z,t)R(\sigma)Y_z .$$ Evaluating both
sides at $ w $ and noting that $ R(\sigma)Y_z(w) = Y_z(z) = Y_w(w) $ we
see that $ c(z,t) $ is a constant as long as $ |z| =1 $ and hence $ k(z,t)
= c(t) Y_z $ for all $ z \in S^{2n-1}.$ Now the hypothesis $ \delta_r T =
T\delta_r $ translates into the homogeneity
$ k(rz,r^2t) = r^{-2n-2}k(z,t).$ Therefore, for any $ z \in \C^n $
$$ k(z,t) = |z|^{-2n-2} k\left(\frac{z}{|z|},\frac{t}{|z|^2}\right) =  |z|^{-2n-2}
c\left(\frac{t}{|z|^2}\right)Y_{\frac{z}{|z|}}.$$ Let $ P_j $ stand for the solid
harmonic satisfying $ P_j(z) = |z|^{p+q} Y_j\left(\frac{z}{|z|}\right).$
Expanding $ k(z,t) $ in terms of
$ Y_j $ and noting that $ (Y_{\frac{z}{|z|}}, Y_j)_{\CH_{p,q}}
 = Y_j(\frac{z}{|z|}) = |z|^{-p-q} P_j(z) $ we get
$$ k_j(z,t) = |z|^{-2n-p-q-2}c\left(\frac{t}{|z|^2}\right)P_j(z) = g_j(z,t)P_j(z) .$$
Note that $ g_j(z,t) $ is a radial function satisfying the homogeneity
condition $ \delta_r g_j = r^{-2n-p-q-2}g_j.$

Let $ \beta $ be any linear functional on $ \CH_{p,q}.$ Then we have
$$ \beta(Tf) = \sum_{j=1}^{d(p,q)} c_j T_jf = \sum_{j=1}^{d(p,q)} c_j f*k_j.$$
The operator $ T_j $ is a Fourier multiplier operator with multiplier
$M_j(\lambda) = W_{\lambda}(k_j^\lambda)$ which can be calculated using
Hecke-Bochner formula. As $ k_j^\lambda(z) = g_j^\lambda(z)P_j(z) $ it
follows that for $ \lambda > 0 $
$$ W_\lambda(P_jg_j^\lambda) = G_\lambda(P_j)\left( \sum_{k=p}^\infty
R_{k-p}^\lambda(g_j^\lambda)P_k(\lambda) \right) $$
where
$$ R_k^\lambda(g_j^\lambda) =  \frac{\Gamma(k-p+1)\Gamma(n)}{\Gamma(k+q+n)}
 \int_{\C^{n+p+q}} g_j^\lambda(z)\varphi_{k,\lambda}^{n+p+q-1}(z) dz.$$
As $ g_j^\lambda(rz) = r^{-2n-p-q}g_j^{\lambda r^2}(z) $ we have
$ R_k^{\lambda r^2}(g_j^{\lambda r^2}) = r^{-p-q} R_k^\lambda
(g_j^\lambda) $ and hence $  R_k^\lambda(g_j^\lambda) = R_k^1(g_j^1)
\lambda^{-(p+q)/2}.$

Therefore, we have shown that the multiplier $ M_j(\lambda) $ is of the form
$$ M_j(\lambda) = G_\lambda(P_j) \left( \sum_{k=p}^\infty c_{k-p} P_k(\lambda)
\right) \lambda^{-(p+q)/2}.$$ Finally, the hypothesis on
$ M_j(\lambda)P_k(\lambda) $ shows that $ c_{k-p} = (2k+n)^{-(p+q)/2} $ for
all $ k \geq p.$ It can be checked that $ G_\lambda(P_j)P_k(\lambda) = 0 $
for $ k < p $ and
consequently $ M_j(\lambda) = G_\lambda(P_j)H(\lambda)^{-(p+q)/2}, $
which proves Theorem 1.4.

We now turn our attention to proofs of Theorems 1.5 and 1.6. Recall that 
for $ P
\in \CH_{p,q} $ we have defined $ R_P $ as the (higher order) Riesz transform
with multiplier $ G_\lambda(P)H(\lambda)^{-\frac{p+q}{2}}.$ As defined earlier 
$ O(P_0) $ stands for the orbit of $ P_0 $ under the action of $ U(n) .$ In 
other words, $ O(P_0) $ is the set of all $ \rho(\sigma)P_0 $ as $\sigma $ 
ranges over $U(n).$ 

\begin{prop} In order to prove dimension-free estimate for
$ R_P, P \in O(P_0) $ it is enough to consider $ P_0(z) = z_j^p \bar{z}_k^q.$
(Here we assume $ n \geq 2 $).
\end{prop}
\begin{proof} To prove the proposition we make use of the operators
$ \rho(\sigma) $ introduced earlier. For $ f \in L^p(\H^n),
\|\rho(\sigma)f\|_p = \|f\|_p $ and hence we need to show that
$$ \int_{\H^n} |\rho(\sigma)(R_P\rho(\sigma^*)f)(z,t)|^p dz dt \leq C_p
\int_{\H^n} |f(z,t)|^p dz dt.$$ From the theorem of Stone-von Neumann we know
that for every $ \sigma \in U(n) $
$$ \pi_\lambda(\sigma.z,t) = \mu_\lambda(\sigma)^*\pi_\lambda(z,t)
\mu_\lambda(\sigma)$$
where $\mu_\lambda(\sigma)$ are certain unitary operators on $L^2(\R^n).$ In
view of this a simple calculation shows that
$$ \pi_\lambda(\rho(\sigma)f) = \mu_\lambda(\sigma)\hat{f}(\lambda)
\mu_\lambda(\sigma)^* $$
and hence
$$ \pi_\lambda(\rho(\sigma)R_P\rho(\sigma^*)f) = \hat{f}(\lambda)
\mu_\lambda(\sigma)G_\lambda(P)\mu_\lambda(\sigma)^*
H(\lambda)^{-\frac{p+q}{2}} $$
where we have made use of the fact that $\mu_\lambda(\sigma) $ commutes with
$ H(\lambda).$ But
$$ \mu_\lambda(\sigma)G_\lambda(P)\mu_\lambda(\sigma)^* =
G_\lambda(\rho(\sigma)P) $$ and the proposition is proved by choosing $ \sigma
$ such that $ \rho(\sigma)P(z) = z_j^p \bar{z}_k^q $ which is possible for
any $ P \in O(P_0)$.
\end{proof}

The proof of the above proposition actually shows that if $ R_{P} $ is bounded 
on $ L^p(\H^n) $ then so is $ R_{(\rho(\sigma)P)}.$ As $ R $ is a unitary 
representation of $ U(n) $ on $  \CH_{p,q}, $ 
any $ P \in \CH_{p,q} $ is a finite linear combination of elements from the 
orbit of $ P_0(z).$ This proves Theorem 1.5.

We continue with our proof of Theorem 1.6. As $ j \neq k $ are arbitrary, 
we can very well assume that
$ P_0(z) = z_1^p \bar{z}_2^q .$ The operator $ G_\lambda(P) $ associated to 
it by the  Weyl correspondence is explicitly known and
given by $ A_2(\lambda)^q A_1(\lambda)^{*p},$ see
Geller \cite{G} or \cite{T2}.
Thus it is enough to consider the Riesz transform with this multiplier.
Without loss of generality we can assume that $ q \geq p.$ Since
$ A_2(\lambda)A_1(\lambda)^* $ commutes with $ H(\lambda) $ (as can be easily
checked by testing against the Hermite functions $ \Phi_\alpha^\lambda $) the
multiplier can be written as
$$  A_2(\lambda)^q A_1(\lambda)^{*p} H(\lambda)^{-\frac{p+q}{2}} =
A_2(\lambda)^{q-p}H(\lambda)^{-\frac{q-p}{2}}A_2(\lambda)^p
A_1(\lambda)^{*p} H(\lambda)^{-p}.$$
Furthermore, using the commutativity again, we have
$$ A_2(\lambda)^p A_1(\lambda)^{*p} H(\lambda)^{-p}
= (A_2(\lambda) A_1(\lambda)^*H(\lambda)^{-1})^p $$
which can be put in the form
$$ \left(H(\lambda)^{-1/2}
A_2(\lambda) A_1(\lambda)^*H(\lambda)^{-1/2}\right)^p .$$
It is now clear that the operator with multiplier
$ A_2(\lambda)^p A_1(\lambda)^{*p} H(\lambda)^{-p}$ satisfies a dimension-free
estimate since the first order Riesz transforms with multipliers
$ H(\lambda)^{-1/2} A_2(\lambda) $ and $  A_1(\lambda)^*H(\lambda)^{-1/2} $
satisfy such estimates. Here we have made use of the fact that the multiplier
$ H(\lambda)^{-1/2} A_2(\lambda) $ corresponds to the adjoint of the Riesz
transform whose multiplier is $ A_2(\lambda)^*H(\lambda)^{-1/2}.$
Therefore, Theorem 1.6 will be proved once we have

\begin{thm} For  any integer $ m \geq 1 $ the Riesz transform with multiplier
$ A_2(\lambda)^m  H(\lambda)^{-m/2} $ satisfies dimension-free estimates.
\end{thm}
\begin{proof} Writing
$$  A_2(\lambda)^m H(\lambda)^{-m/2} = \prod_{j=0}^{m-1}
H(\lambda)^{j/2}A_2(\lambda) H(\lambda)^{-(j+1)/2} $$
we only need to treat multipliers of the form
$$ H(\lambda)^{j/2}A_2(\lambda) H(\lambda)^{-(j+1)/2} $$
for $ j \geq 1 $ which we do using an induction argument. Recalling that
$$ H(\lambda) = \frac{1}{2} \sum_{j=1}^n\left(
A_j(\lambda)A_j^*(\lambda)+A_j^*(\lambda)A_j(\lambda)\right) $$ and the commutation
relation
$$ A_j(\lambda)A_j(\lambda)^* - A_j^*(\lambda)A_j(\lambda) = -4\lambda I $$
we can show that
$$ H(\lambda) A_j(\lambda) =  A_j(\lambda) H(\lambda)+ 4\lambda A_j(\lambda).$$
In view of this
$$  H(\lambda)^{j/2}A_2(\lambda) H(\lambda)^{-(j+1)/2} =
 H(\lambda)^{(j-2)/2} H(\lambda)A_2(\lambda) H(\lambda)^{-(j+1)/2} $$
reduces to
$$  H(\lambda)^{(j-2)/2} A_2(\lambda) H(\lambda)^{-(j-1)/2} + 4
 H(\lambda)^{(j-2)/2} A_2(\lambda) H(\lambda)^{-(j-1)/2}
\lambda H(\lambda)^{-1}.$$
If we can show that the operator with multiplier $ \lambda H(\lambda)^{-1}$
satisfies dimension-free estimates then we can use an induction on $ j $ to
complete the proof of the above theorem. Thus the proof is completed by
proving the following result.
\end{proof}

\begin{thm} The multiplier $ \lambda H(\lambda)^{-1}$ defines a singular
integral operator on $ \H^n $ which satisfies dimension-free bounds on all $ L^p $ spaces,
$ 1< p < \infty.$
\end{thm}
\begin{proof} The boundedness of this operator is well known, see Folland
\cite{F} and Stein \cite{S3}. Indeed, the above multiplier corresponds to
$ T \CL^{-1} $ where $ T = \frac{\partial}{\partial t} $ and is given by
convolution with the kernel
$ K(z,t) = T \varphi_0(z,t) $ where
$$ \varphi_0(z,t) = c_n (|z|^4+t^2)^{-n/2} $$ is the fundamental solution
for the sublaplacian found by Folland \cite{F}. Here the constant $ c_n $ is
given by
$$ c_n^{-1} = n(n+2) \int_{\H^n}(1+t^2+|z|^4)^{-(n+4)/2} |z|^2 dz dt.$$ It
is clear that $ K(z,t) $ is an odd kernel which is homogeneous of degree
$(-2n-2) $ and hence defines a bounded operator on $ L^p(\H^n),
1< p < \infty.$ The norm of this operator can be estimated using a result
of Christ \cite{C} as presented in  Strichartz \cite{St} (Lemma 3.1). This
shows that the norm on $ L^p $ is bounded by
$$ C_p \int_{\C^n} |K(z,1)| dz $$
where $ C_p $ depends only on $ p.$ Thus we are left with proving the
inequality
$$ n c_n \int_{\C^n} \left(1+|z|^4\right)^{-n/2-1} dz  \leq C $$
for some constant $ C $ independent of $n.$ This follows from the next lemma.
\end{proof}

\begin{lem} There is a constant $ C $ independent of $ n $ such that
$$ \frac{ \int_{\C^n} (1+|z|^4)^{-n/2-1} dz }
{(n+2) \int_{\H^n}(1+t^2+|z|^4)^{-(n+4)/2} |z|^2 dz dt } \leq C $$
for all values on $ n.$
\end{lem}
\begin{proof} Integrating in polar coordinates, the above ratio reduces to
$$ \frac{ \int_0^\infty (1+r^4)^{-n/2-1} r^{2n-1}dr }
{2(n+2) \int_0^\infty \int_0^\infty (1+t^2+r^4)^{-(n+4)/2} r^{2n+1} dr dt }.$$
By a simple change of variables the integral on the numerator  can be seen
to be the beta integral
$$ \frac{1}{4}\int_0^\infty (1+t)^{-n/2-1}t^{n/2-1} dt  =
\frac{ \Gamma(1)\Gamma(n/2)}{4\Gamma(n/2+1)}.$$
Similarly, the integral in the denominator becomes the product
$$ \bigg( \int_0^\infty (1+t^2)^{-3/2}dt \bigg) \bigg(\int_0^\infty
(1+u^4)^{-(n+4)/2} u^{(2n+1)} du\bigg).$$
As before,
$$
\int_0^\infty (1+t^2)^{-3/2}dt = \frac{\Gamma(1/2)\Gamma(1)}
{2\Gamma(3/2)} $$
and
$$ \int_0^\infty (1+u^4)^{-(n+4)/2} u^{(2n+1)} du=
\frac{ \Gamma(3/2)\Gamma((n+1)/2)}
{4 \Gamma((n+4)/2)} .$$
Thus the ratio of the two integrals is given by
$$ \frac{ \Gamma(n/2)}{4\Gamma(n/2+1)} \frac{8\Gamma(n/2+2)}
{\Gamma(1/2)\Gamma((n+1)/2)}.$$
Simplifying this we see that the ratio in the
lemma is given by
$$ \pi^{-1/2} \frac{\Gamma(n/2)}{\Gamma((n+1)/2)} $$
which is clearly a bounded function of $ n.$
\end{proof}

\section{Riesz transforms for Hermite and Laguerre expansions}
\setcounter{equation}{0}

In this section we consider Riesz transforms associated to Hermite and
Laguerre expansions. First we can deduce Theorem 1.3 from the corresponding
result for the reduced Heisenberg group. We fix $ \lambda = 1 $ and simply
write $ W $ and $ G $ instead of $ W_1 $ and $ G_1.$ Consider
the Weyl multiplier $ T_P $ defined by
$$ W(T_Pf) = W(f)G(P)H^{-(p+q)/2} $$ where $ P \in \CH_{p,q}.$
The Weyl multipliers can be considered as higher order Riesz transforms for
the special Hermite operator. Recall that the special Hermite operator $ L $
is defined by $ \CL(e^{it}f(z)) = e^{it}Lf(z).$ The spectral decomposition
of this operator is given by
$$ Lf = (2\pi)^{-n} \sum_{k=0}^\infty f\times \varphi_k^{n-1} $$
where $ \varphi_k^{n-1}(z) $ are Laguerre functions of type $ (n-1) $ and
$ \times $
stands for the twisted convolution:
$$ f\times g(z) = \int_{\C^n} f(z-w) g(w) e^{\frac{i}{2}\Im(z \cdot \bar{w})}
dw.$$

When $ P(z) = z_j $ (resp. $\bar{z}_j $) it follows from the work of
Geller  that $ G(P) = Z_j(1)$ (resp. $ G(P) = \overline{Z_j}(1))$
and hence $ T_{z_j} = Z_j(1)L^{-1/2} $ and $ T_{\bar{z}_j}=
\overline{Z_j}(1))L^{-1/2} $ are
the Riesz transforms for the special
Hermite  expansions. More generally if $ P(z) = \sum_{|\alpha|=p,|\beta| =q}
c_{\alpha,\beta} z^\alpha \bar{z}^\beta $ then $ G(P) $ is obtained by
replacing $ z_j $ and $ \bar{z}_j $ by $ Z_j(1) $ and $ \overline{Z_j}(1) $
respectively. Though the following result can be deduced from results on the
reduced Heisenberg group by considering functions of the form $ e^{it}f(z) $
we can give a simple direct proof based on the Hecke-Bochner formula for the
Weyl transform.

\begin{thm} For every $ P \in \CH_{p,q} $ the Riesz transforms $ T_P $ are
bounded on $ L^p(\C^n), 1 < p < \infty $ and are also weak type $ (1,1).$
The first order Riesz transforms satisfy dimension-free bounds.
\end{thm}
\begin{proof} The proof is very similar to the one given in \cite{T0} for
the cases $ P(z) = z_j $ and $ P(z) = \bar{z}_j.$ Hence we will only give a
sketch of the proof. The Hecke-Bochner formula for the special Hermite
projections (see Geller \cite{G} and \cite{T2}) it follows that
$$ G(P)P_k = c_n(p,q) W(P\varphi_{k-p}^{n+p+q-1}) $$ (see 2.6.22 in \cite{T2}) and hence
$$ G(P)H^{-(p+q)/2} = c_n(p,q) W(PK_m) $$
where $ m = p+q $ and $ K_m $ is the kernel defined by
$$ K_m(z) = \sum_{k=0}^\infty (2k+2p+n)^{-m/2}\varphi_k^{n+p+q-1}(z).$$ On the
other hand using the relations
$$ W(Z_jf) = W(f)A_j,~~~~ W(\bar{Z}_jf) = W(f)A_j^* $$ we can obtain
$$ W(L^{-m/2}P(Z,\bar{Z})f) = W(f)G(P)H^{-m/2} $$
so that a comparison with the above formula shows that
$$ L^{-m/2}P(Z,\bar{Z})f(z) = f\times(PK_m)(z) .$$ The kernel $ K_m $ can
be expressed in terms of the heat kernel for $ L $ on $ \C^{n+p+q} $ and hence
estimated. It turns out that  $ L^{-m/2}P(Z,\bar{Z})$ is an oscillatory
singular integral operator. Hence by the results of Ricci-Stein \cite{RS}
and Chanillo-Christ \cite{CC} it is bounded on $ L^p(\C^n) $ for $ 1 < p <
\infty $ and of weak type (1,1). By considering the adjoint we get the
desired result.
\end{proof}

We now deduce some results for Riesz transforms associated to multiple
Laguerre expansions. In terms of special Hermite functions the spectral
decomposition of $ L $ takes the form
$$  f(z) = \sum_{\alpha,\beta} (f,\Phi_{\alpha,\beta})\Phi_{\alpha,\beta}(z) .$$
When $ f $ is a polyradial function, i.e., $
f(\sigma z) = f(z)$ for all diagonal
matrices $ \sigma $ in $ U(n),$ the special Hermite expansion of $ f $ reduces to
a multiple Laguerre expansion. More generally, when $ f $ is a polyradial
function the special Hermite expansion of $ g(z) = z^m f(z) $ where $ m \in
\N^n $ takes the form
$$  g(z) = \sum_{\alpha} \left( \int_{\C^n} g(z) \overline{\Phi_{\alpha,\alpha+m}}
(w) dw \right) \Phi_{\alpha,\alpha +m}(z).$$  Indeed, we have the formula
(see Theorem 1.3.5 in \cite{T0})
$$ \Phi_{\alpha,\alpha +m}(z) = (2\pi)^{-n/2} \left( \frac{\alpha!}
{(\alpha +m)!}\right)^{1/2} 2^{-|m|/2}(-i)^{|m|} z^m \prod_{j=1}^n
\varphi_{\alpha_j}^{m_j}(z_j) $$ and hence in the expansion only the term
$(g,\Phi_{\alpha,\alpha+m})$ survives. Note that
$$  \left( \int_{\C^n} g(z) \overline{\Phi_{\alpha,\alpha+m}}
(w) dw \right) \Phi_{\alpha,\alpha +m}(z)$$
$$ = (2\pi)^{-n} \frac{\alpha!}
{(\alpha +m)!}(\frac{z}{2})^m \left( \int_{\C^n} f(w)\prod_{j=1}^n
\varphi_{\alpha_j}^{m_j}(w_j) |w_j|^{2m_j+1} dw \right) \prod_{j=1}^n
\varphi_{\alpha_j}^{m_j}(z_j).$$

We now define multiple Laguerre functions of type $ m $ by
$$ \Psi_{\alpha}^m(r) = \left(\frac{2^{-|m|}\alpha!}
{(\alpha+m)!}\right)^{1/2} \prod_{j=1}^n \varphi_{\alpha_j}^{m_j}(r_j)$$
where $ r =(r_j), r_j \geq 0.$ Here
$ \varphi_{\alpha_j}^{m_j}(r_j) $ stands for $\varphi_{\alpha_j}^{m_j}
(z_j) $ with $ |z_j| = r_j.$ These functions form an orthonormal basis  for
$ L^2(\R_+^n, d\mu_m) $ where
$$ d\mu_m(r) = \prod_{j=1}^n r_j^{2m_j+1}dr_j.$$
The multiple Laguerre expansion of type $ m $ of a function $ f
\in L^2(\R_+^n, d\mu_m) $ is related to the special Hermite expansion of
$ g(z) = z^m f(r), r_j = |z_j|.$
Indeed,
$$ \sum_{\alpha} (g,\Phi_{\alpha,\alpha+m})\Phi_{\alpha,\alpha+m}(z)
= z^m  \sum_{\alpha} (f, \Psi_{\alpha}^m) \Psi_{\alpha}^m(r).$$
Consider the Riesz transform $ T_{\bar{z}_j} = \overline{Z}_j(1)L^{-1/2} $
applied to the function $ g.$ Since
$$ \overline{Z}_j(1)\Phi_{\alpha,\beta} = i(2\alpha_j)^{1/2}
\Phi_{\alpha-{e_j},\beta} $$ it follows that
$$ T_{\bar{z}_j}g(z) = i\sum_{\alpha} (2\alpha_j)^{1/2}(2|\alpha|+n)^{-1/2}
 (g,\Phi_{\alpha,\alpha+m})\Phi_{\alpha-e_j,\alpha+m}(z) $$
which can be
expressed in terms of $ f $ as
$$   T_{\bar{z}_j}g(z)= i z^{m+e_j} \sum_{\alpha} (2\alpha_j)^{1/2}
(2|\alpha|+n)^{-1/2} (f, \Psi_{\alpha}^m)
\Psi_{\alpha-e_j}^{m+e_j}(r).$$
Thus we have ( with $ r_j = |z_j|$)
$$  T_{\bar{z}_j}g(z) = i z^m \frac{z_j}{|z_j|}R_{j,m}f(r)$$
where
$$ R_{j,m}f(r) = r_j  \sum_{\alpha}(2\alpha_j)^{1/2}(2|\alpha|+n)^{-1/2}
(f, \Psi_{\alpha}^m) \Psi_{\alpha-e_j}^{m+e_j}(r) $$
are  the Riesz transforms associated to Laguerre expansions of type $ m.$
These transforms have been studied by Nowak and Stempak \cite{NS1}
and \cite{NS2}. From the above relation we obtain

\begin{thm} The Laguerre Riesz transforms $ R_{j,m} $ are bounded on
$ L^p(\R_+^n, d\mu_m) $ if and only $ T_{\bar{z}_j} $ satisfy the weighted norm
inequality
$$ \int_{\C^n} |T_{\bar{z}_j}g(z)|^p \prod_{j=1}^n |z_j|^{m_j(2-p)} dz
\leq C \int_{\C^n} |g(z)|^p  \prod_{j=1}^n |z_j|^{m_j(2-p)} dz $$
for all functions $ g $ which are $ m-$homogeneous, i.e. $ g(z) = z^m f(z) $
where $ f $ is polyradial.
\end{thm}

When $ m = 0 $ we do have the boundedness of $  T_{\bar{z}_j} $ on
$ L^p(\C^n), 1 < p < \infty $ and hence the Riesz transforms $ R_{j,0} $ are
all bounded on $L^p(\R_+^n, d\mu_m) .$ However, for general $ m $ we can use
the result of \cite{NS1} which says that  $ R_{j,m} $ are bounded on
$ L^p(\R_+^n, d\mu_m) $ to obtain the following result.

\begin{cor} For any $ m \in \N^n $ we have the weighted norm inequality
$$ \int_{\C^n} |T_{\bar{z}_j}g(z)|^p \prod_{j=1}^n |z_j|^{m_j(2-p)} dz
\leq C \int_{\C^n} |g(z)|^p  \prod_{j=1}^n |z_j|^{m_j(2-p)} dz $$
for all $ m-$homogeneous $ g \in L^p(\C^n), 1 < p < \infty .$
\end{cor}

Since the Riesz transforms $T_{\bar{z}_j} $ are (oscillatory) singular
integral operators it is natural to expect that they will satisfy weighted
norm inequalities  provided the weight is coming from $ A_p $ class. In the
above inequality the weight function $ w(z) = \prod_{j=1}^n |z_j|^{m_j(2-p)} $
belongs to $ A_p $ only if $ 2(1-\frac{1}{2+m_j}) < p < 2(1+\frac{1}{m_j}).$
Thus
for functions with certain homogeneity we can expect weighted norm
inequalities under weaker assumptions. It would be interesting to
characterise all weight functions for which the Riesz transforms $ T_P $
satisfy weighted norm inequalities. We plan to return to this problem in the
future.

Interchanging the roles of $ T_{\bar{z}_j} $ and  $ R_{j,m} $ and using
Theorem 4.1 we can obtain the following

\begin{cor}  For any $ m \in \N^n $ we have the weighted norm inequality
$$ \int_{\R_+^n} |R_{j,m}f(r)|^p \prod_{j=1}^n |r_j|^{m_j(p-2)} d\mu_{m}
\leq C_p \int_{\R_+^n} |f(r)|^p  \prod_{j=1}^n |r_j|^{m_j(p-2)} d\mu_{m} $$
for all $ f \in L^p(\R_+^n), 1 < p < \infty $ where $ C_p $ is independent
of $ n $ and $ m.$
\end{cor}

It would be interesting to see if such dimension-free, type-free
estimates are true for
general multiple Laguerre expansions of type $ m $ studied by \cite{NS1} where
$ m_j$'s are not necessarily integers. Riesz transforms for  multiple
Laguerre expansions of a different kind have been studied by Gutierrez et al.
in \cite{GIT}. When $ m_j$'s are half integers they have shown that the
Riesz transforms satisfy dimension-free bounds.

\begin{center}
{\bf Acknowledgments}

\end{center}
The authors are extremely  thankful to the referee for his thorough and
careful reading of the manuscript and for pointing out a couple of
mathematical errors and a host of typos.  The work of the first author is supported by the All India Council For Technical Education.  The work of the second author is supported  by J. C. Bose Fellowship from
the Department of Science and Technology (DST) and also by a grant from UGC
via DSA-SAP.

\end{document}